\documentclass[12pt]{amsart}
\usepackage{amsmath}
\usepackage{amsthm}
\usepackage{amssymb}
\usepackage{amsfonts,mathrsfs}
\usepackage{stmaryrd}
\usepackage{amsxtra}
\usepackage{epsfig}
\usepackage{verbatim}
\usepackage{color}

\theoremstyle{plain}
\newtheorem{theorem}[equation]{Theorem}
\newtheorem{proposition}[equation]{Proposition}
\newtheorem{lemma}[equation]{Lemma}

\newtheorem{definition}[equation]{Definition}
\theoremstyle{remark}

\numberwithin{equation}{section}

\include{header}

\newcommand{\wB}{\mathbf{B}_{\lambda}}
\newcommand{\Bmu}{\mathbf{B}_{\mu}}
\newcommand{\Ba}{\mathbf{B}_{\lambda_\alpha}}
\newcommand{\uD}{\mathbb{D}}
\begin{document}

\bibliographystyle{alpha}

\title[$L^p$ Regularity of Weighted Projections]{$L^p$ Regularity of Some Weighted Bergman Projections on the Unit Disc}
\author{Yunus E. Zeytuncu}
%\thanks{}
\address{Department of Mathematics, Texas A\&M University, College Station, Texas 77843}
\email{zeytuncu@math.tamu.edu}
\date{}
\begin{abstract} 
We show that weighted Bergman projections, corresponding to weights of the form $M(z)(1-|z|^2)^{\alpha}$ where $\alpha>-1$ and $M(z)$ is a radially symmetric, strictly positive and at least $C^2$ function on $\overline{\uD}$, are $L^p$ regular.
\end{abstract}
\maketitle

\section{Introduction}

Let $\mathbb{D}$ denote the unit disc in $\mathbb{C}^1$ and $dA(z)$ denote the standard Lebesgue measure on $\mathbb{C}^1$. Let $\lambda(r)$ be a strictly positive and continuous function on $[0,1)$. We consider $\lambda(r)$ as a radially symmetric weight on $\uD$ by setting $\lambda(z):=\lambda(|z|)$ and denote the space of square integrable functions with respect to the area element $\lambda(z)dA(z)$ by $L^2(\lambda)$. It is clear that $L^2(\lambda)$ is a Hilbert space with the inner product defined by
\begin{equation*}
\left<f,g\right>_{\lambda}=\int_{\mathbb{D}}f(z)\overline{g(z)}\lambda(z)dA(z)
\end{equation*}
and the norm defined by
\begin{equation*}
\left|\left|f\right|\right|_{\lambda}^2=\int_{\mathbb{D}}|f(z)|^2\lambda(z)dA(z).
\end{equation*}

The closed subspace of holomorphic functions in $L^2(\lambda)$ is denoted by $A^2(\lambda)$. The orthogonal projection operator between these two spaces is called \textit{the weighted Bergman projection} and denoted by $\wB$, i.e., 
\begin{equation*}
\wB: L^2(\lambda) \to A^2(\lambda).
\end{equation*}

The Riesz representation theorem indicates that $\wB$ is an integral operator. The kernel of this integral operator is called \textit{the weighted Bergman kernel} and denoted by $B_{\lambda}(z,w)$, i.e. for any $f\in L^2(\lambda)$
\begin{equation*}
\wB f(z)=\int_{\uD}B_{\lambda}(z,w)f(w)\lambda(w)dA(w).
\end{equation*}

The monomials $\{z^n\}_{n=0}^{\infty}$ form an orthogonal basis of $A^2(\lambda)$ and the weighted Bergman kernel is given by the following sum:
\begin{equation*}
B_{\lambda}(z,w)=\sum_{n=0}^{\infty}a_n(z\bar w)^n,
\text{ where }a_n = \frac{1}{\int_{\uD}|z|^{2n}\lambda(z)dA(z)}. 
\end{equation*}
The coefficients $a_n$ are called \textit{the Bergman coefficients} of the weight $\lambda$. 

For $1<p<\infty$, we use the standard notation $L^p(\lambda)$ and $A^p(\lambda)$ to denote the respective Banach spaces of $p$-integrable functions on $\uD$ and we use $||. ||_{p,\lambda}$ to denote the norm on these spaces.\\

Let us consider the weights defined by $\lambda_{\alpha}(r)=(1-r^2)^{\alpha}$ for $\alpha>-1$, where we set $z=re^{i\theta}$. The Bergman theory for this family of weights are well investigated and can be found in \cite{Hakan}. 

In particular, the Bergman coefficients of these weights are computed explicitly and the following explicit expression for the weighted kernel is obtained:
\begin{equation*}
B_{\lambda_\alpha}(z,w)=\frac{c_{\alpha}}{\left(1-z\overline{w}\right)^{2+\alpha}},
\end{equation*}
where $c_{\alpha}$ is a constant that only depends on $\alpha$.

Furthermore, this explicit expression for the kernel and Schur's lemma together prove the following theorem. 
\begin{theorem}\label{first}
For $\alpha>-1$, the weighted Bergman projection $\mathbf{B}_{\lambda_\alpha}$ is bounded from $L^p(\lambda_\alpha)$ to $A^p(\lambda_\alpha)$ for any $1<p<\infty$.
\end{theorem}

\begin{proof}
See page 12 of \cite{Hakan} and also \cite{Zhubook} and \cite{ForelliRudin}. \\
\end{proof}

The purpose of this note is to extend this theorem to more general weights in the following setup. Let $M(r)$ be a strictly positive and at least $C^2$ function on $[0,1]$. Without loss of generality, we assume that $M(1)=1$. Consider the radially symmetric weight defined by
$$\mu(z)=M(|z|)(1-|z|^2)^{\alpha}$$
on $\mathbb{D}$, for some $\alpha>-1$. By the general theory (see \cite{Duren04} and \cite{ForelliRudin}), there exists the weighted Bergman projection operator $\mathbf{B}_{\mu}:L^2(\mu) \to A^2(\mu)$, which is an integral operator with the weighted Bergman kernel $B_{\mu}(z,w)$, where
\begin{equation*}
B_{\mu}(z,w)=\sum_{n=0}^{\infty}b_n(z\bar w)^n,
\text{ and }b_n=\frac{1}{\int_{\uD}|z|^{2n}\mu(z)dA(z)}. 
\end{equation*}
But in this case, it is not easy (unless $M$ is a simple function) to compute the coefficients $b_n$ to get an explicit expression for the weighted kernel and therefore, Schur's lemma is not directly applicable in this case.\\

Nevertheless, we prove the analog of Theorem \ref{first} for $\Bmu$, without referring to an explicit expression for the kernel or Schur's lemma.

\begin{theorem}\label{second}
The weighted Bergman projection $\mathbf{B}_{\mu}$ is bounded from $L^p(\mu)$ to $A^p(\mu)$ for any $1<p<\infty$.
\end{theorem}

The proof is in two steps; first relating $\Bmu$ to $\Ba$ by a coefficient multiplier operator and then showing that this coefficient multiplier operator is bounded.\\

For the rest of the note, we denote the boundary of $\mathbb{D}$ by $b\mathbb{D}$ and we write $A\lesssim B$ to mean $A\leq cB$ for some constant $c$ that is clear in context. We also use the Szeg\"o projection $\mathbf{T}:L^2(b\mathbb{D}, d\theta)\to H^2$, where $d\theta$ is the arc length on the unit circle and $H^p$ is the Hardy space of order $p$. We refer to \cite{Duren04} for definitions and standard facts about the Szeg\"o projection and Hardy spaces.\\

This article is a part of my PhD dissertation at The Ohio State University. I thank J.D. McNeal, my advisor, for introducing me to this field and helping me with various points. I also thank the anonymous referee for helpful comments.\\

\section{Coefficient Multipliers and Norm Convergence}

In this section, before giving the details of the proof of Theorem \ref{second}, we recall a few facts about coefficient multipliers. See \cite{Vukotic99} and \cite{Duren04} for general account.\\

Let $X$ be a Banach space of holomorphic functions on $\mathbb{D}$. Any $f\in X$ has Taylor series expansion $$f(z)=\sum_{n=0}^{\infty}f_nz^n.$$
\begin{definition}
A sequence of complex numbers $\{t_n\}$ is called a coefficient multiplier from $X$ to $X$ and denoted by
$\{t_n\} \in (X, X)$ if for any function $f\in X$, 
\begin{equation*}
t(f)(z):=\sum_{n=0}^{\infty}t_nf_nz^n~\text{ is also in }X.
\end{equation*}
\end{definition}

It is a fairly general question to characterize the coefficient multipliers on an arbitrary Banach space $X$ and there is no full answer to this question.
\begin{definition}\label{partialsum}
For a holomorphic function $f$ on $\mathbb{D}$ and $N\in\mathbb{N}$, let $S_Nf$ denote 
the Taylor polynomial of $f$ of degree $N$, i.e., $S_Nf(z)=\sum_{n=0}^{N}f_nz^n.$
\end{definition}

If $X$ has the property that for any $f\in X$ the sequence of Taylor polynomials $\left\{S_Nf\right\}$ converges to $f$, then a sufficient condition for coefficent multipliers can be formulated as follows. 

\begin{proposition}\label{sufficient}
Let $(X,||.||)$ be a Banach space of holomorphic functions on $\mathbb{D}$ such that for every $f\in X$ the sequence $\left\{S_Nf\right\}$ of Taylor polynomials converges to $f$ in the norm of $X$. Then any sequence of \textbf{bounded variation} is a coefficient multiplier from $X$ to $X$.
\end{proposition}

\begin{definition} A sequence of complex numbers $\{t_n\}$ is said to be of bounded variation if $|t_0|+\sum_{n=1}^{\infty}|t_n-t_{n-1}|~\text{ is finite.}$\end{definition}

Proposition \ref{sufficient} appears in \cite[Proposition 3.7]{Vukotic99}. It follows from summation by parts and we repeat its proof for completeness.
\begin{proof}%[Proof of Proposition \ref{sufficient}]
Since the Taylor polynomials converge, for any given $f\in X$ and $\epsilon>0$ there exists an $N$ such that for any $k>N$, $$\left|\left|\sum_{n=k}^{\infty}f_nz^n\right|\right|<\epsilon.$$ Let $\{t_n\}$ be the sequence of bounded variation and $|t_0|+\sum_{n=1}^{\infty}|t_n-t_{n-1}|\leq K$. Summation by parts and bounded variation hypothesis give
\begin{align*}
\left|\left|\sum_{n=k}^{\infty}t_nf_nz^n\right|\right|&=\left|\left|\sum_{n=k}^{\infty}(t_{n+1}-t_n)\sum_{j=n+1}^{\infty}f_jz^j+t_k\sum_{n=k}^{\infty}f_nz^n\right|\right|\\&\leq \left[|t_k|+\sum_{n=k}^{\infty}|t_{n+1}-t_n|\right]\epsilon \\ &\leq K\epsilon
\end{align*}
This shows that $t(f)(z)=\sum_{n=0}^{\infty}t_nf_nz^n$ is in $X$ and finishes the proof.
\end{proof}

\vskip 1cm
In order to use this proposition in the proof of Theorem \ref{second}, we have to check whether Taylor polynomials converge in $A^{p}(\mu)$. This turns out to be true even in a more general form.
\begin{proposition}\label{Taylorconvergence} 
For $1<p<\infty$ and any integrable radial weight $ \lambda(r)$, the Taylor series of every function in $A^p(\lambda)$ converges in norm.
\end{proposition}
\noindent In particular, the claim is true for $A^{p}(\lambda_\alpha)$ and $A^{p}(\mu)$. The statement for $A^p(\lambda_\alpha)$ is in \cite{Zhu91}. The general case is obtained by just imitating the proof in \cite{Zhu91}.

\begin{proof}% \textit{Proposition \ref{Taylorconvergence}}\\
This is done in three steps.

\noindent \textit{Step One.} The holomorphic polynomials are dense in $A^p(\lambda)$.

\noindent For any $f \in A^p(\lambda)$ and for any $0< \rho<1$, define $f_{\rho}(z)=f(\rho z)$. Each $f_{\rho}$ is holomorphic in a larger disc and the Taylor polynomials of each $f_{\rho}$ converges uniformly on $\mathbb{D}$ and hence in $A^p(\lambda)$. Therefore it is enough to show that 
$$\lim_{\rho\to1^-}||f-f_{\rho}||_{p,\lambda}=0.$$
For any holomorphic $f$, the averages $$M_p^p(r,f)=\frac{1}{2\pi}\int_{0}^{2\pi}|f(re^{i\theta})|^pd\theta$$ are well defined and non-decreasing functions of $r$ (see \cite[page 26]{Duren04}). Moreover $$M_p^p(r,f_{\rho})=M_p^p(\rho r,f)\leq M_p^p(r,f).$$ Since $f\in A^p(\lambda)$ and 
\begin{align*}
||f||_{p,\lambda}^p=\int_0^1r\lambda(r)M_p^p(r,f)dr.
\end{align*}
$M_p^p(r,f)$ is integrable with respect to the weight $r\lambda(r)dr$.

\noindent On the other hand, $f_{\rho}\to f$  pointwise on $\mathbb{D}$ as $\rho \to 1^-$ so by the Lebesgue dominated convergence theorem $\lim_{\rho \to 1^-}M_p^p(r,f-f_{\rho})=0$. We also have 
$$M_p^p(r,f-f_{\rho})\leq 2^p \left(M_p^p(r,f)+M_p^p(r,f_{\rho}\right)\leq 2^{p+1}M_p^p(r,f).$$

\noindent Therefore again the Lebesgue dominated convergence theorem implies
\begin{align*}
\lim_{\rho\to1^-}||f-f_{\rho}||_{p,\lambda}^p&=\lim_{\rho\to1^-}\int_0^1r\lambda(r)M_p^p(r,f-f_{\rho})dr\\
&=\int_0^1r\lambda(r)\lim_{\rho\to1^-}M_p^p(r,f-f_{\rho})dr\\
&=0.
\end{align*}
\noindent This finishes the first step.\\

\noindent \textit{Step Two.} We show that the operator norms of $S_N$'s (defined in Definition \ref{partialsum}) are uniformly bounded. For this we need a well-known result about the Szeg\"o projection. Let $\mathbf{T}:L^2(b\mathbb{D}, d\theta)\to H^2$ denote the Szeg\"o projection. By using the fact that $\mathbf{T}$ is also bounded from $L^p(b\mathbb{D}, d\theta)$ to $H^p$ for any $1<p<\infty$, one can prove (see \cite[page 27]{Duren04}) that  there exists $C>0$ , independent of $N$ and $h$, such that
\begin{equation}
\int_{0}^{2\pi}|S_Nh(e^{i\theta})|^pd\theta\leq C\int_{0}^{2\pi}|h(e^{i\theta})|^pd\theta
\end{equation}
for any $h\in H^p$. The proof is only to note that $\overline{S_Nf(e^{i\theta})}=e^{-iN\theta}T\left(e^{iN\theta}\overline{f(e^{i\theta})}\right)$ which is clear for $f$ a polynomial, and follows in general since polynomials are dense in $H^p$.

Now we calculate the operator norms of $S_N$'s. For given $f\in  A^p(\lambda)$,

\begin{align*}
||S_Nf||_{p,\lambda}^p&=\int_{\mathbb{D}}|S_Nf(z)|^p\lambda(z)dA(z)\\
&=\int_0^1r\lambda(r)dr\int_{0}^{2\pi}|S_Nf(re^{i\theta})|^pd\theta\\
&\leq C \int_0^1r\lambda(r)dr\int_{0}^{2\pi}|f(re^{i\theta})|^pd\theta~ \text{ since } f_r\in H^p\\
&=C\int_{\mathbb{D}}|f(z)|^p\lambda(z)dA(z)\\
&=C ||f||_{p,\lambda}^p.
\end{align*}
\noindent This implies that $\sup_{N}||S_N||_{op}\leq C$ and finishes the second step.\\

\noindent \textit{Step Three.} Next, we show that $\lim_{N\to \infty}||S_Nf-f||_{p,\lambda}=0$ for any $f\in A^p(\lambda)$. 

\noindent Given $f$ and $\epsilon>0$, by the first step there exists a polynomial $Q$ such that $||Q-f||_{p,\lambda}^p<\epsilon$. Then
\begin{align*}
||S_Nf-f||_{p,\lambda}^p&\leq ||S_Nf-S_NQ||_{p,\lambda}^p+||S_NQ-Q||_{p,\lambda}^p+||Q-f||_{p,\lambda}^p\\
&\leq (C+1)\epsilon + ||S_NQ-Q||_{p,\lambda}^p.
\end{align*}
Note that $S_NQ=Q$ for large enough $N$ and therefore for sufficiently large $N$,
$$||S_Nf-f||_{p,\lambda}^p\leq (C+1)\epsilon.$$
Since this is true for any $\epsilon>0$ we get  $\lim_{N\to \infty}||S_Nf-f||_{p,\lambda}=0$. This finishes the last step and the proof of the proposition.\\
\end{proof}

\section{Proof of Theorem \ref{second}}

In this section, we prove Theorem \ref{second} by using Propositions \ref{sufficient} and \ref{Taylorconvergence}. Recall that $a_n$'s are the Bergman coefficients of $(1-|z|^2)^{\alpha}$ and $b_n$'s are the Bergman coefficients of $\mu$. Let $\mathcal{R}$ denote the coefficient multiplier operator for the sequence $\left\{ \frac{b_n}{a_n} \right\}$. The following identity relates the two Bergman projections:
\begin{equation}\label{relation}
\mathbf{B}_{\mu}f(z)=\mathcal{R}\left[\mathbf{B}_{\lambda_\alpha}\left(fM\right)\right](z).
\end{equation}

\noindent Indeed, for any $f \in L^2(\mu)$,
\begin{align*}
\mathbf{B}_{\mu}f(z)=&\int_{\mathbb{D}}\sum_{n=0}^{\infty}b_n(z\bar{w})^nf(w)\mu(w)dA(w)=\sum_{n=0}^{\infty}b_nz^n\int_{\mathbb{D}}\bar{w}^nf(w)\mu(w)dA(w)&\\
=&\sum_{n=0}^{\infty}a_nz^n\frac{b_n}{a_n}\int_{\mathbb{D}}\bar{w}^nf(w)\mu(w)dA(w)&\\
=&\mathcal{R}\left[\sum_{n=0}^{\infty}a_nz^n\int_{\mathbb{D}}\bar{w}^nf(w)\mu(w)dA(w)\right]&\\
=&\mathcal{R}\left[\mathbf{B}_{\lambda_\alpha}\left(fM\right)\right](z).&\\
\end{align*}
Here we change the order of integration and summation but this doesn't cause any problems. We can truncate the summation, which is equivalent to looking at the Taylor polynomials of $\mathbf{B}_{\mu}f$ and $\mathbf{B}_{\lambda_\alpha}\left(fM\right)$, and take limit by using Proposition \ref{Taylorconvergence}. Now, it suffices to prove that the multiplier operator $\mathcal{R}$ is bounded from $A^{p}(\mu)$ to $A^{p}(\mu)$ (actually, we have to show that $\mathcal{R}$ is bounded from $A^{p}(\lambda_{\alpha})$ to $A^{p}(\mu)$ but since $M$ is of class $C^2$ and thus bounded; the inclusion map $i:A^{p}(\lambda_{\alpha}) \to A^{p}(\mu)$ is bounded). By the closed graph theorem it is enough to show that $\mathcal{R}\left(f\right) \in A^{p}(\mu)$ for any $f\in A^{p}(\mu)$. Moreover, Proposition \ref{sufficient} implies that it is enough to show that the sequence $\left\{\frac{b_n}{a_n}\right\}$ is of bounded variation.

It is immediate that the sequence $\left\{\frac{b_n}{a_n}\right\}$ is bounded from below and above. Moreover, a direct computation gives that $$\lim_{n\to \infty}\frac{b_n}{a_n}=\lim_{n\to \infty}\frac{\int_0^1r^{2n+1}(1-r^2)^{\alpha}dr}{\int_0^1r^{2n+1}\mu(r)dr}=M(1)^{-1}.$$ We quantify this computation to get that the sequence $\left\{\frac{b_n}{a_n}\right\}$ is indeed of bounded variation.\\

\begin{lemma}\label{multiplierestimate}
$\left|\frac{b_n}{a_n}-\frac{b_{n-1}}{a_{n-1}}\right|\lesssim\frac{1}{n^2}$, i.e., the sequence $\left\{\frac{b_n}{a_n}\right\}$ is of bounded variation and therefore $\mathcal{R}$ is bounded from $A^{p}(\mu)$ to $A^{p}(\mu)$.
\end{lemma} 

\begin{proof}

First, we consider the difference between elements of the sequence $\left\{\frac{b_n}{a_n}\right\}$. Here all the integrals are taken with respect to $r$ and from 0 to 1.

\begin{align*}
\frac{b_n}{a_n}-\frac{b_{n-1}}{a_{n-1}}&=\frac{\int r^{2n+1}(1-r^2)^{\alpha}}{\int r^{2n+1}\mu(r)}-\frac{\int r^{2n-1}(1-r^2)^{\alpha}}{\int r^{2n-1}\mu(r)}\\
&=\frac{\int r^{2n+1}(1-r^2)^{\alpha}\int r^{2n-1}\mu(r)-\int r^{2n-1}(1-r^2)^{\alpha}\int r^{2n+1}\mu(r)}{\int  r^{2n+1}\mu(r)\int r^{2n-1}\mu(r)}\\
&=:\frac{B(n)}{A(n)}.\\
\end{align*}
We can rewrite the numerator as
\begin{align*}
B(n)&=\int r^{2n+1}(1-r^2)^{\alpha}\int r^{2n-1}(1-r^2)\mu(r)-\int r^{2n-1}(1-r^2)^{\alpha+1}\int r^{2n+1}\mu(r)\\
&=\int r^{2n+1}\left[\left(1-M(r)\right)(1-r^2)^{\alpha}\right]\int r^{2n-1}(1-r^2)\mu(r)\\
&\hskip 2cm-\int r^{2n-1}\left[\left(1-M(r)\right)(1-r^2)^{\alpha+1}\right]\int r^{2n+1}\mu(r)\\
&=:B_1(n)-B_2(n).\\
\end{align*}

\noindent Next, we integrate $B_1(n)$ and $B_2(n)$ by parts twice to obtain
\begin{align*}
B_1(n)&=\frac{1}{(2n+2)2n}\int r^{2n+2}\left[\left(1-M(r)\right)(1-r^2)^{\alpha}\right]'\int r^{2n}\left[M(r)(1-r^2)^{\alpha+1}\right]'\\
&=:\frac{1}{(2n+2)2n}C_1(n)C_2(n)\\
&\\
B_2(n)&=\frac{1}{2n(2n+1)}\int r^{2n+1}\left[\left(1-M(r)\right)(1-r^2)^{\alpha+1}\right]''\int r^{2n+1}\left[M(r)(1-r^2)^{\alpha}\right]\\
&=:\frac{1}{2n(2n+1)}C_3(n)C_4(n).
\end{align*}
Here, $C_1,C_2,C_3,C_4$ denote the respective integrals. Note that we don't get any boundary terms after integration by parts since $M$ is of class $C^2$ on $[0,1]$ and $M(1)=1$.\\

\noindent In order to finish the proof, it suffices to show that
\begin{equation*}\sup_{n}\left\{n^2\left|\frac{B_1(n)}{A(n)}\right|\right\}~\text{ and }~\sup_{n}\left\{n^2\left|\frac{B_2(n)}{A(n)}\right|\right\}~\text{ are finite.}
\end{equation*}
Thus, it is enough to show that 
\begin{equation*}\sup_{n}\left\{\left|\frac{C_1(n)C_2(n)}{A(n)}\right|\right\}~\text{ and }~\sup_{n}\left\{\left|\frac{C_3(n)C_4(n)}{A(n)}\right|\right\}~\text{ are finite.}
\end{equation*}\\

\noindent We start with the first one.
\begin{align*}
\frac{C_1(n)C_2(n)}{A(n)}&=\frac{\int r^{2n+2}\left[\left(1-M(r)\right)(1-r^2)^{\alpha}\right]'\int r^{2n}\left[M(r)(1-r^2)^{\alpha+1}\right]'}{\int  r^{2n+1}M(r)(1-r^2)^{\alpha}\int r^{2n-1}M(r)(1-r^2)^{\alpha}}\\
&\to\frac{\left[\left(1-M(r)\right)(1-r^2)^{\alpha}\right]'\left[M(r)(1-r^2)^{\alpha+1}\right]'}{M(r)(1-r^2)^{\alpha}M(r)(1-r^2)^{\alpha}} \vert_{r=1} ~\text{ as } n\to \infty\\
&=2(\alpha+1)^2M'(1).
\end{align*}
This shows that the first supremum is indeed finite. Note that the condition $M(1)=1$ is used here.\\

\noindent We argue the same way for the second one.
\begin{align*}
\frac{C_3(n)C_4(n)}{A(n)}&=\frac{\int r^{2n+1}\left[\left(1-M(r)\right)(1-r^2)^{\alpha+1}\right]''\int r^{2n+1}\left[M(r)(1-r^2)^{\alpha}\right]}{\int  r^{2n+1}M(r)(1-r^2)^{\alpha}\int r^{2n-1}M(r)(1-r^2)^{\alpha}}\\
&\to\frac{\left[\left(1-M(r)\right)(1-r^2)^{\alpha+1}\right]''\left[M(r)(1-r^2)^{\alpha}\right]}{M(r)(1-r^2)^{\alpha}M(r)(1-r^2)^{\alpha}} \vert_{r=1} ~\text{ as } n\to \infty\\
&=2(\alpha+2)(\alpha+1)M'(1)-2(1+\alpha)\left(1-M(1)\right).
\end{align*}
This shows that the second supremum is indeed finite. Again, note that the condition $M(1)=1$ is used here. This finishes the proof Lemma \ref{multiplierestimate}.

\end{proof}

Since Lemma \ref{multiplierestimate} is established, we conclude the proof of Theorem \ref{second}.

\end{document}